\documentclass[reqno, 12pt]{amsart}

\usepackage{amsmath, amssymb, amsthm} 
\usepackage{hyperref} 
\usepackage[dvipdfmx]{graphicx}

\newenvironment{parts}[0]{%
  \begin{list}{}%
    {\setlength{\itemindent}{0pt}
     \setlength{\labelwidth}{1.5\parindent}
     \setlength{\labelsep}{.5\parindent}
     \setlength{\leftmargin}{2\parindent}
     \setlength{\itemsep}{0pt}
     }%
   }%
  {\end{list}}
\newcommand{\Part}[1]{\item[\upshape#1]}

\allowdisplaybreaks[2] 

\newcommand{\N}{\mathbb{N}} 
\newcommand{\Z}{\mathbb{Z}} 
\newcommand{\Q}{\mathbb{Q}} 
\newcommand{\QB}{\overline{\mathbb{Q}}} 
\newcommand{\R}{\mathbb{R}}

\renewcommand{\P}{\mathfrak{P}} 
\newcommand{\Df}{\mathfrak{D}} 

\newcommand{\p}{\mathfrak{p}}

\newcommand{\g}{\gamma}

\newcommand{\eg}{\textit{e.g.}}
\newcommand{\f}{\frac}

\renewcommand{\t}{\text} 
\renewcommand{\l}{\left} 
\renewcommand{\r}{\right} 

\renewcommand{\O}{\mathcal{O}} 
\renewcommand{\d}{\delta}

\newtheorem{Theorem}{Theorem}[section] 
\newtheorem{Proposition}[Theorem]{Proposition} 
\newtheorem{Lemma}[Theorem]{Lemma} 
\newtheorem{Corollary}[Theorem]{Corollary} 
\newtheorem{Fact}[Theorem]{Fact} 

\theoremstyle{definition}
\newtheorem{Definition}[Theorem]{Definition} 
\newtheorem{Remark}[Theorem]{Remark}

\newtheorem*{acknowledgement}{Acknowledgements}
\theoremstyle{plain}

\title[Relative Northcott numbers]{Relative Northcott numbers for the weighted Weil heights} 
\author[M. Okazaki]{Masao Okazaki} 
\address{Graduate School of Mathematics \\ 
Kyushu University \\ 
Motooka 744, Nishi-ku, Fukuoka 819-0395 \\ 
Japan.} 
\email{m-okazaki@math.kyushu-u.ac.jp} 
 

\begin{document}
\begin{abstract} 
It is fundamental in number theory to calculate lower bounds for height functions. 
Grizzard studied lower bounds for the Weil height in a relative setting. 
Vidaux and Videla introduced the Northcott number for a set $A\subset\overline{\mathbb{Q}}$. 
It bounds the Weil height on $A$ from below, outside the zero-height points and the finitely many small-height points. 
Pazuki, Technau, and Widmer introduced the weighted Weil heights. 
These heights generalize both the absolute and relative Weil heights. 
In this paper, we introduce a relative version of the Northcott number related to the weighted Weil height. 
We also give a field extension whose Northcott number equals a given positive number. 
The work is a relative version of the previous work of the author and Sano on the Northcott numbers for the weighted Weil heights. 
\end{abstract}

\maketitle

\section{Introduction}\label{Intro} 

Let $h:\QB\rightarrow\R_{\geq0}$ be the absolute logarithmic Weil height (see, \eg, \cite[p.16]{BG}). 
For $\g\in\R$ and $a\in\QB$, we set 
\[ 
h_\g(a):=\deg(a)^\g h(a), 
\] 
where $\deg(a):=[\Q(a):\Q]$. 
The function $h_\g$ is called {\it $\g$-weighted Weil height}, introduced in \cite{PTW}. 
We note that $h_0$ (resp. $h_1$) is the absolute (resp. relative) logarithmic Weil height. 

\begin{Definition}[{\cite[Section 1]{BZ} or \cite[Section 1]{PTW}}] 
We say that a subset $A\subset\QB$ has the {\it $\g$-Bogomolov property} (resp. {\it $\g$-Northcott property}) if the set $\l\{ a\in A \mid 0<h_\g(a)<C \r\}$ (resp. $\l\{ a\in A \mid h_\g(a)<C \r\}$) is finite for some (resp. for all) $C>0$. 
We denote $\g$-Bogomolov property (resp. $\g$-Northcott property) by $\g$-(B) (resp. $\g$-(N)) for short. 
\end{Definition} 

\begin{Remark} 
An algebraic number $a\in\QB$ satisfies that $h_\g(a)=0$ if and only if $a$ is $0$ or a root of unity 
(see, \eg, \cite[p.17, Theorem 1.5.9]{BG}). 
\end{Remark} 

\begin{Remark}\label{interval} 
For a subset $A\subset\QB$, the property $\g$-(B) of $A$ immediately follows from $\g$-(N) of $A$. 
Hence if we set 
\begin{align*} 
I_B(A)&:=\l\{ \g\in\R \mid A\t{ has }\g\t{-(B)} \r\} \t{ and} \\
I_N(A)&:=\l\{ \g\in\R \mid A\t{ has }\g\t{-(N)} \r\}, 
\end{align*} 
then $I_B(A)\supset I_N(A)$ holds. 
By definition, we know that both $I_B(A)$ and $I_N(A\setminus\{\t{root of unity}\})$ are interval $(\g,\infty)$ or $[\g,\infty)$ for some $\g\in\R\cup\{-\infty\}$. 
As we will see in Lemma \ref{weak North}, the equality $\inf I_B(A)=\inf I_N(A\setminus\{\t{root of unity}\})$ holds. 
\end{Remark} 

That $A\subset\QB$ has $\g$-(B) means that the values of $h_\g$ on $A$ are bounded from below by an absolute positive constant, outside the zero-height points. 
Any number field has $\g$-(N) for all $\g\in\R$. 
This is an immediate consequence of the Northcott theorem (see, \eg, \cite[p.25, Theorem 1.6.8]{BG}). 
On the one hand, there are two previous works on infinite extensions of $\Q$ having $\g$-(B) or $\g$-(N) (see \cite{OS} and \cite{PTW}). 
However, there are many examples of infinite extensions having $0$-(B) or $0$-(N) (see, \eg, \cite{AD}, \cite{BZ}, \cite{Ha}, \cite{Schi}, or \cite{Wi}). 
Seeing the definition of $\g$-(B), we want to set 
\[
{\rm Nor}_\g(A):=\inf\l\{ C>0 \mid \#\l\{ a\in A \mid h_\g(a)<C \r\}=\infty \r\}. 
\] 
The non-negative number ${\rm Nor}_\g(A)$ is called the {\it $\g$-Northcott number} of $A$, introduced in \cite{OS} and \cite{VV}. 
Note that $A$ has $\g$-(B) (resp. $\g$-(N)) if and only if ${\rm Nor}_\g(A\setminus\{\t{root of unity}\})>0$ (resp. ${\rm Nor}_\g(A)=\infty$). 
It is natural to try to construct a field whose $\g$-Northcott number equals a given positive number. 
The problem was first dealt with in \cite[Theorem 3]{PTW} and later solved in \cite[Theorem 4.1]{OS}. 
In this paper, we focus on the problem in a ``relative'' setting. 
We introduce the following: 

\begin{Definition} 
Let $L/K$ be an extension of subfields of $\QB$. 
We say that $L/K$ is a {\it relative $\g$-Bogomolov extension} (resp. {\it relative $\g$-Northcott extension}) if the set $L\setminus K$ has $\g$-(B) (resp. $\g$-(N)). 
We denote relative $\g$-Bogomolov (resp. relative $\g$-Northcott) by $\g$-(RB) (resp. $\g$-(RN)) for short. 
\end{Definition} 

\begin{Remark} 
The $0$-(RB) extension is the {\it relative Bogomolov extension} introduced in \cite{Gri}. 
\end{Remark} 

\begin{Remark}\label{nontrivial} 
If a field $L\subset\QB$ has $\g$-(B) (resp. $\g$-(N)), then $L/K$ is $\g$-(RB) (resp. $\g$-(RN)) for any field $K\subset L$. 
We say that a $\g$-(RB) (resp. $\g$-(RN)) extension $L/K$ is {\it trivial} if $L\subset\QB$ has $\g$-(B) (resp. $\g$-(N)). 
By definition, if $K$ has $\g$-(B) (resp. $\g$-(N)), then any $\g$-(RB) (resp. $\g$-(RN)) extension $L/K$ is trivial. 
Thus we should restrict our attention to the case that $K$ does not have $\g$-(B) (resp. $\g$-(N)). 
Here note that the Lehmer conjecture asserts that $\QB$ has $1$-(B). 
On the one hand, we know that the set $\QB\setminus\{\t{root of unity}\}$ has $\g$-(N) for all $\g>1$ (see \cite[Remark 2.6]{OS}). 
In addition, for all $\g<0$ (resp. $\g\leq 0$), we will see in Section \ref{negative} that any $\g$-(RB) (resp. $\g$-(RN)) extension is trivial. 
Hence we deal with $\g$-(RB) (resp. $\g$-(RN)) only in the case $0\leq\g<1$ (resp. $0<\g\leq1$). 
\end{Remark} 

Now we set $I_B(L/K):=I_B(L\setminus K)$, $I_N(L/K):=I_N(L\setminus K)$, and ${\rm Nor}_\g(L/K):={\rm Nor}_\g(L\setminus K)$. 
We call ${\rm Nor}_\g(L/K)$ {\it relative $\g$-Northcott number} of $L/K$. 
We also want to construct an extension $L/K$ such that $K$ does not has $\g$-(B) and $\t{Nor}_\g(L/K)$ equals a given positive number. 
Hence our result is the following: 

\begin{Theorem}\label{implicit} 
Let $c\in\R_{>0}$. 
Then we can construct an extension $L/K$ of subfields of $\QB$ satisfying the equality $I_N(K)=(1,\infty)$ and the following condition {\rm (1)}, {\rm (2)}, or {\rm (3)}, respectively. 
\begin{parts} 
\Part{(1)} 
\hypertarget{main1} 
$I_B(L/K)=I_N(L/K)=(\g,\infty)$ for $\g\in[0,1)$.

\Part{(2)} 
\hypertarget{main2} 
$I_B(L/K)=[\g,\infty)\supsetneq(\g,\infty)=I_N(L/K)$ with ${\rm Nor}_\g(L/K)=c$ for $\g\in[0,1)$. 

\Part{(3)} 
\hypertarget{main3} 
$I_B(L/K)=I_N(L/K)=[\g,\infty)$ for $\g\in(0,1]$. 

\end{parts} 
\end{Theorem} 

\begin{Remark} 
Regarding Remark \ref{nontrivial}, we note that the condition $I_N(K)=(1,\infty)$ implies that $K$ does not have $\g$-(B) (resp. $\g$-(N)) for all $\g\in\R_{<1}$ (resp. $\g\in\R_{\leq1}$) by Remark \ref{interval}. 
\end{Remark}

\section{Preparations}\label{Pre} 

This section is devoted to giving technical lemmata to prove Theorem \ref{implicit}. 
The discussions are greatly based on those in \cite{OS}, \cite{PTW}, and \cite{Wi}. 
First, we recall (a special case of) Silverman's inequality. 

\begin{Fact}[{\cite[Theorem 2]{Si84}}]\label{Siiq} 
Let $K$ be a number field. 
Assume that $a\in \QB$ satisfies the inequality $d:=[K(a):K]>1$. 
We set $L:=K(a)$. 
Then the inequality 
\[
h(a)\geq \f{1}{2(d-1)}\l(\f{\log(N_{K/\Q}(D_{L/K}))}{d[K:\Q]}-\log(d)\r) 
\] 
holds, where $N_{K/\Q}$ is the usual norm and $D_{L/K}$ is the relative discriminant ideal of the extension $L/K$. 
\end{Fact} 

For a number field $K$, we denote by $\O_K$ the ring of integers of $K$. 
To give a more explicit lower bound for the Weil height by using Fact \ref{Siiq}, we must estimate $N_{K/\Q}(D_{L/K})$ from below. 
In \cite[Proof of Theorem 4]{Wi}, Widmer gave the following nice tool to do it. 

\begin{Lemma}\label{ramify} 
Let $p$ be a prime number and $L/K$ be an extension of number fields. 
Assume that $p$ does not ramify in $K$ and that any prime ideal of $\O_K$ lying above $p$ ramifies totally in $L$. 
Then we have 
\[
p^{[K:\Q]([L:K]-1)} \mid N_{K/\Q}(D_{L/K}). 
\]
\end{Lemma} 

\begin{proof} 
Let $\Df_{L/K}$ be the different ideal of $L/K$. 
We employ the following two facts. 
\begin{parts} 
\Part{(2.1)}\hypertarget{2.1}{} 
Let $\P\subset\O_L$ be a prime ideal and $e\in\N$ be the ramification index of $\P$ over $K$. 
Then it holds that $\P^{e-1} \mid \Df_{L/K}$ (see, \eg, \cite[p.199, (2.6)]{Ne}). 

\Part{(2.2)}\hypertarget{2.2}{} 
It holds that $D_{L/K}=N_{L/K}(\Df_{L/K})$ (see, \eg, \cite[p.201, (2.9)] {Ne}). 

\end{parts} 
By the assumption that $p$ does not ramify in $K$, the prime decomposition of $p\O_K$ in $\O_K$ is a form of 
\[
p\O_K=\p_1\cdots\p_s, 
\label{2.3}\tag{2.3}
\] 
where $\p_i\neq\p_j$ if  $i\neq j$. 
Now for each $1\leq i\leq s$, there exists a prime ideal $\P_i\subset\O_L$ such that the prime decomposition of $\p_i\O_L$ in $\O_L$ is 
\[
\p_i\O_L=\P_i^{[L:K]}. 
\label{2.4}\tag{2.4}
\]
By (\hyperlink{2.1}{2.1}), we know that $\P_i^{[L:K]-1} \mid \Df_{L/K}$ for each $1\leq i\leq s$. 
Thus we have 
\[
(\P_1\cdots\P_s)^{[L:K]-1} \mid \Df_{L/K}. 
\] 
On the one hand, by (\hyperlink{2.2}{2.2}), it holds that 
\[
N_{K/\Q}(D_{L/K})=N_{K/\Q}(N_{L/K}(\Df_{L/K}))=N_{L/\Q}(\Df_{L/K}). 
\]
Therefore, we have 
\[
N_{L/\Q}\l((\P_1\cdots\P_s)^{[L:K]-1}\r) \mid N_{L/\Q}(\Df_{L/K})=N_{K/\Q}(D_{L/K}). 
\label{2.5}\tag{2.5} 
\] 
Now we know that 
\begin{align*} 
N_{L/\Q}\l((\P_1\cdots\P_s)^{[L:K]-1}\r) 
&=\l(N_{L/\Q}\l(\P_1^{[L:K]}\cdots\P_s^{[L:K]}\r)\r)^{\f{[L:K]-1}{[L:K]}} \\
&=\l(N_{L/\Q}((\p_1\cdots\p_s)\O_L)\r)^{\f{[L:K]-1}{[L:K]}} && \text{by }(\ref{2.4}) \\
&=\l(N_{L/\Q}(p\O_L)\r)^{\f{[L:K]-1}{[L:K]}} && \text{by }(\ref{2.3}) \\
&=\l(p^{[L:K][K:\Q]}\r)^{\f{[L:K]-1}{[L:K]}} \\
&=p^{[K:\Q]([L:K]-1)}. 
\end{align*} 
Combining (\ref{2.5}), we have completed the proof. 
\end{proof} 

For a number field $K$ and a prime ideal $\p\subset\O_K$, we say that $X^n+a_{n-1}X^{n-1}+\cdots+a_0\in\O_K[X]$ is a $\p$-Eisenstein polynomial if $a_i\in\p$ for all $0\leq i\leq n-1$ and $a_0\notin\p^2$. 

\begin{Fact}[{\eg, \cite[p.133, Theorem 24 (a)]{FT}}]\label{Eisenstein} 
Let $L/K$ be an extension of number fields and $\p\in\O_K$ be a prime ideal. 
Then any $\p$-Eisenstein polynomial is irreducible in $K[X]$. 
Furthermore, the following two conditions are equivalent. 
\begin{parts} 
\Part{(1)} 
$\p$ ramifies totally in $L$. 
\Part{(2)}  
$L=K(\rho)$, where $\rho$ is a root of some $\p$-Eisenstein polynomial $g(X)\in\O_K[X]$. 
\end{parts} 
\end{Fact}

\begin{Corollary}\label{Lower bound} 
Let $d$, $p$, and $q$ be prime numbers with $p<q$. 
Assume that positive integers $e_1, r_1, \ldots, e_j, r_j$ satisfy that $p\nmid\prod_{n=1}^j(e_nr_n)$ and $q\nmid\prod_{n=1}^j(e_nr_n)$.
We set $K:=\Q(r_n^{1/e_n} \mid 1\leq n\leq j)$ and $L:=K((p/q)^{1/d})$. 
Then the inequality 
\[ 
h(a) > \f{\log(p)}{d}-\f{\log(d)}{2(d-1)} 
\] 
holds for all $a\in L\setminus K$. 
\end{Corollary} 
\begin{proof} 
By \cite[Lemma 4.1]{Fi} and \cite[p.97, Theorem 85]{Hi}, we know that $p$ and $q$ do not ramify in $K$. 
Thus, by Fact \ref{Eisenstein}, any prime ideal $\p\subset\O_K$ lying above $p$ (resp. $q$) ramifies totally in $L$ since $L=K((pq^{d-1})^{1/d})$ (resp. $L=K((p^{d-1}q)^{1/d})$) holds and $X^d-pq^{d-1}\in\O_K[X]$ (resp. $X^d-p^{d-1}q\in\O_K[X]$) is a $\p$-Eisenstein polynomial. 
Therefore we have 
\[
p^{[K:\Q](d-1)} \mid N_{K/\Q}(D_{L/K}) \ \t{ and } \ 
q^{[K:\Q](d-1)} \mid N_{K/\Q}(D_{L/K}) 
\] 
by Lemma \ref{ramify}. 
Combining the assumption that $p<q$, we get the inequality 
\[
2\log(p)
< \log(pq)
\leq \f{\log(N_{K/\Q}(D_{L/K}))}{[K:\Q](d-1)}. 
\label{2.6}\tag{2.6} 
\] 
Now note that $K(a)=L$ holds since $a\in L\setminus K$ and 
\begin{align*} 
[L:K]
&=\deg(X^d-pq^{d-1}) && \t{by Fact \ref{Eisenstein}} \\
&=d 
\end{align*} 
is a prime number. 
Thus we conclude that 
\begin{align*} 
h(a) 
&\geq \f{1}{2(d-1)}\l(\f{\log(N_{K/\Q}(D_{L/K}))}{d[K:\Q]}-\log (d)\r) && \text{by Fact }\ref{Siiq} \\
&> \f{\log(p)}{d}-\f{\log(d)}{2(d-1)} && \t{by } (\ref{2.6}). 
\end{align*} 
This completes the proof.  
\end{proof}

\begin{Fact}[{\cite[Lemma 6]{PTW}}]\label{Northnumber} 
Let $\g\in\R$ and $A\subset\QB$. 
Assume that a nest sequence $A_0\subsetneq A_1\subsetneq A_2\subsetneq\cdots$ of subsets of $A$ satisfies that 
\begin{parts} 
\Part{(1)} 
$A_i$ has $\g$-{\rm(N)} for all $i\in\Z_{\geq 0}$ and 

\Part{(2)} 
$A=\bigcup_{i\in\Z_{\geq0}}A_i$. 

\end{parts} 
Then we have 
\[
{\rm Nor}_\g(A)=\liminf_{i\rightarrow\infty}\inf h_\g(A_i\setminus A_{i-1}). 
\]
\end{Fact}

\begin{Fact}[{\cite[Proposition 2.3]{OS}}]\label{weak North} 
Let $\g\in\R$ and $A\subset\QB$. 
If $A$ has $\g${\rm -(B)}, then $A\setminus\{\t{root of unity}\}$ also has $\d${\rm -(N)} for all $\d>\g$. 
\end{Fact}

Throughout the rest of the paper, we denote by $\N$ the set of positive integers. 

\begin{Lemma}\label{Amo} 
The field $K:=\Q( 2^{1/3^j} \mid j\in\N )$ satisfies the equality $I_N(K)=(1,\infty)$. 
\end{Lemma} 

\begin{proof} 
By Fact \ref{weak North}, it is sufficient to prove that $K$ has $1$-(B) but not $1$-(N). 
By \cite[Theorem 1.3]{Amo}, $K$ has $1$-(B) . 
On the one hand, since we have the inequality $h_1(2^{1/3^j})\leq\log(2)$ for all $j\in\N$, we know that $K$ does not have $1$-(N). 
\end{proof}

\section{Proof of Theorem \ref{implicit}} 

In this section, we prove Theorem \ref{implicit}. 
In fact, we prove the following more explicit one.

\begin{Theorem}\label{explicit} 
Let $\g\in[0,1)$, $c\in\R_{>0}$, and $K:=\Q(2^{1/3^j} \mid j\in\N)$. 

\begin{parts}
\Part{\rm (A)}\hypertarget{main'1}{} 
We set 
\[ 
f(x):=\f{x\log(x)}{2(x-1)}+cx. 
\] 
Take strictly increasing sequences of prime numbers $(d_i)_{i\in\N}$, $(p_i)_{i\in\N}$, and $(q_i)_{i\in\N}$ satisfying the inequalities $\min\{d_1,p_1\}>3$, 
\begin{align*} 
& \exp\l(f(d_i)\r)\leq p_i<q_i<2p_i \leq 4\exp\l(f(d_i)\r), \t{ and} \\
& \max\l\{ d_i, 4f(d_i) \r\}<\exp\l(f(d_{i+1})\r) 
\end{align*} 
for all $i\in\N$. 
We set $L:=K((p_i/q_i)^{1/p_i} \mid i\in\N)$. 
Then $L/K$ satisfies the conditions {\rm (\hyperlink{main2}{2})} in Theorem \ref{implicit} of the case $\g=0$.

\Part{\rm (B)}\hypertarget{main'2}{} 
We let $f(x)$ be $1/\log(x)$, $c$, or $\log(x)$. 
Take strictly increasing sequences of prime numbers $(d_i)_{i\in\N}$, $(p_i)_{i\in\N}$, and $(q_i)_{i\in\N}$ satisfying the inequalities $\min\{d_1,p_1\}>3$, 
\begin{align*} 
& \f{d_i^\g\log(d_i)}{2(d_i-1)}<f(d_i), \\
& \exp(f(d_i)d_i^{1-\g})\leq p_i<q_i<2p_i \leq 4\exp(f(d_i)d_i^{1-\g}), \t{ and} \\
& \max\{ d_i, 4\exp(f(d_i)d_i^{1-\g}) \}<\exp(f(d_{i+1})d_{i+1}^{1-\g})
\end{align*} 
for all $i\in\N$. 
We set $L:=K((p_i/q_i)^{1/d_i} \mid i\in\N)$. 
\begin{parts} 
\Part{(1)} 
If $f(x)=1/\log(x)$, then $L/K$ satisfies the conditions {\rm (\hyperlink{main1}{1})} in Theorem \ref{implicit}. 
\Part{(2)} 
If $f(x)=c$, then $L/K$ satisfies the conditions {\rm (\hyperlink{main2}{2})} in Theorem \ref{implicit} of the cases $\g\in(0,1)$. 
\Part{(3)} 
If $f(x)=\log(x)$, then $L/K$ satisfies the conditions {\rm (\hyperlink{main3}{3})} in Theorem \ref{implicit} of the case $\g\in(0,1)$. 
\end{parts}

\Part{\rm (C)}\hypertarget{main'3}{} 
Take strictly increasing sequences of prime numbers $(p_i)_{i\in\N}$ and $(q_i)_{i\in\N}$ satisfying the inequalities $p_1>3$ and 
\[ 
p_i<q_i<2p_i<p_{i+1}
\] 
for all $i\in\N$. 
We set $L:=K((p_i/q_i)^{1/p_i} \mid i\in\N)$. 
Then $L/K$ satisfies the conditions {\rm (\hyperlink{main3}{3})} in Theorem \ref{implicit} of the case $\g=1$. 
\end{parts} 
\end{Theorem}

\begin{Remark} 
We can take the sequences of prime numbers in Theorem \ref{explicit} by the Bertrand–Chebyshev theorem and by letting $d_i$ be large enough. 
\end{Remark}

\begin{proof}[Proof of Thoerem \ref{explicit}] 
By Lemma \ref{Amo}, we know that the equality $I_N(K)=(1,\infty)$ holds. 
We set 
\begin{align*} 
L_0&:=K, \\ 
L_i&:=L_0((p_m/q_m)^{1/d_m} \mid 1\leq m\leq i), \\ 
L_{(0,j)}&:=\Q(2^{1/2^j}), \t{ and} \\ 
L_{(i,j)}&:=L_{(0,j)}((p_m/q_m)^{1/d_m} \mid 1\leq m\leq i) 
\end{align*} 
for each $i\in\N$ and $j\in\Z_{\geq0}$ (we consider $(d_i)_{i\in\N}$ to be $(p_i)_{i\in\N}$ in (\hyperlink{main'3}{3})). 
Note that $L=\bigcup_{i\in\Z_{\geq0}}L_i$ and $L_i=\bigcup_{j\in\Z_{\geq0}}L_{(i,j)}$ hold for each $i\in\Z_{\geq0}$. 

\begin{parts} 
\Part{(\hyperlink{main'1}{A})} 
Take any $a\in L\setminus K$. 
We set $i_0:=\min\l\{ i\in\N \mid a\in L_i \r\}$. 
By the assumption that $a\notin K$, we know that $i_0\geq 1$. 
Fix $j_0\in\N$ such that $a\in L_{(i_0,j_0)}$. 
By the definition of $i_0$, we know that $a\in L_{(i_0,j_0)}\setminus L_{(i_0-1,j_0)}$. 
Applying Corollary \ref{Lower bound} to the extension $L_{(i_0,j_0)}/L_{(i_0-1,j_0)}$, we get the inequalities 
\begin{align*} 
h(a) 
&> \f{\log(p_{i_0})}{d_{i_0}}-\f{\log(d_{i_0})}{2(d_{i_0}-1)} \\ 
&\geq \f{f(d_{i_0})}{d_{i_0}}-\f{\log(d_{i_0})}{2(d_{i_0}-1)} \\ 
&= \f{1}{d_{i_0}}\l(\f{d_{i_0}\log(d_{i_0})}{2(d_{i_0}-1)}+cd_{i_0}\r)-\f{\log(d_{i_0})}{2(d_{i_0}-1)} \\
&= c. 
\end{align*} 
This implies that $\t{Nor}_0(L/K)\geq c$. 
On the one hand, they hold that $(p_i/q_i)^{1/d_i}\notin K$ for all $i\in\N$ and that 
\[ 
h((p_i/q_i)^{1/d_i}) 
= \f{\log(q_i)}{d_i} 
\leq \f{\log(4)}{d_i}+\f{\log(d_i)}{2(d_i-1)}+c 
\rightarrow c
\] 
as $i\rightarrow\infty$. 
These imply that $\t{Nor}_0(L/K)\leq c$. 
Therefore we conclude that $L/K$ is $0$-(RB) with ${\rm Nor}_0(L/K)=c$. 
By Fact \ref{weak North}, the assertion follows.

\Part{(\hyperlink{main'2}{B})} 
We will apply Fact \ref{Northnumber} to $L\setminus K=\bigcup_{i\in\N}(L_i\setminus K)$. 
First, we prove that $L_i\setminus K$ has $\g$-(N) for all $i\in\N$. 
Take any $a\in L_i\setminus K$. 
We set $i_0:=\min\l\{ m\in\{0,1.\ldots,i\} \mid a\in L_m \r\}$. 
By the same discussion as (\hyperlink{main'1}{1}), we get the inequalities 
\begin{align*} 
h(a) 
&> \f{\log(p_{i_0})}{d_{i_0}}-\f{\log(d_{i_0})}{2(d_{i_0}-1)} \\ 
&\geq \min_{1\leq m\leq i}\l\{\f{\log(p_m)}{d_m}-\f{\log(d_m)}{2(d_m-1)}\r\} \\ 
&\geq \min_{1\leq m\leq i}\l\{f(d_m)d_m^{-\g}-\f{\log(d_m)}{2(d_m-1)}\r\} \\
&= \min_{1\leq m\leq i}\l\{d_m^{-\g}\l(f(d_m)-\f{d_m^\g\log(d_m)}{2(d_m-1)}\r)\r\} 
=: C_i. 
\end{align*} 
Note that $C_i$ is a positive constant depending only on $i$. 
Therefore the above inequalities imply that $L_i\setminus K$ has $0$-(B). 
By Fact \ref{weak North}, we know that $L_i\setminus K$ has $\g$-(N). 
Now take any $\d\in(0,1)$. 
We also take any $b\in(L_i\setminus K)\setminus(L_{i-1}\setminus K)=(L_i\setminus L_{i-1})\setminus K$. 
Again fix $j_1\in\N$ such that $b\in L_{(i,j_1)}$. 
Since $[L_{(i,j_1)}:L_{(i-1,j_1)}]=d_i$ (recall the proof of Corollary \ref{Lower bound}) is a prime number and $b\in L_{(i,j_1)}\setminus L_{(i-1,j_1)}$, we know that $L_{(i-1,j_1)}(b)=L_{(i,j_1)}$. 
Thus we have the inequality 
\[ 
\deg(b)\geq d_i. 
\label{3.1}\tag{3.1} 
\] 
Applying Corollary \ref{Lower bound} to the extension $L_{(i,j_1)}/L_{(i-1,j_1)}$, we get the inequalities 
\begin{align*} 
h_\d(b) 
&> d_i^\d\l(\f{\log(p_i)}{d_i}-\f{\log(d_i)}{2(d_i-1)}\r) && \t{by (\ref{3.1})} \\ 
&\geq d_i^\d\l(f(d_i)d_i^{-\g}-\f{\log(d_i)}{2(d_i-1)}\r) \\
&= d_i^{\d-\g}\l(f(d_i)-\f{d_i^\g\log(d_i)}{2(d_i-1)}\r) \\
&\rightarrow 
\begin{cases} 
0 & (\d<\g) \\
0 & (\d=\g, \, f(x)=1/\log(x)) \\
c & (\d=\g, \, f(x)=c) \\
\infty & (\d=\g, \, f(x)=\log(x)) \\
\infty & (\d>\g)
\end{cases} 
\end{align*} 
as $i\rightarrow\infty$. 
By Fact \ref{Northnumber}, we have the inequality 
\[ 
\t{Nor}_\d(L/K) 
\geq 
\begin{cases} 
0 & (\d<\g) \\
0 & (\d=\g, \, f(x)=1/\log(x)) \\
c & (\d=\g, \, f(x)=c) \\
\infty & (\d=\g, \, f(x)=\log(x)) \\
\infty & (\d>\g)
\end{cases} . 
\label{3.2}\tag{3.2} 
\] 
On the one hand, they hold that $(p_i/q_i)^{1/d_i}\notin K$ for all $i\in\N$ and that 
\begin{align*} 
h_\d((p_i/q_i)^{1/d_i}) 
= \f{\log(q_i)}{d_i^{1-\d}} 
&\leq \f{\log(4)}{d_i^{1-\d}}+\f{f(d_i)}{d_i^{\g-\d}} \\
&\rightarrow 
\begin{cases} 
0 & (\d<\g) \\
0 & (\d=\g, \, f(x)=1/\log(x)) \\
c & (\d=\g, \, f(x)=c) \\
\infty & (\d=\g, \, f(x)=\log(x)) \\
\infty & (\d>\g) 
\end{cases} 
\end{align*} 
as $i\rightarrow\infty$. 
These imply the inequality 
\[ 
\t{Nor}_\d(L/K) 
\leq 
\begin{cases} 
0 & (\d<\g) \\
0 & (\d=\g, \, f(x)=1/\log(x)) \\
c & (\d=\g, \, f(x)=c) \\
\infty & (\d=\g, \, f(x)=\log(x)) \\
\infty & (\d>\g) 
\end{cases}. 
\label{3.3}\tag{3.3} 
\] 
Therefore the assertions follow from the inequalities (\ref{3.2}) and (\ref{3.3}).

\Part{(\hyperlink{main'3}{C})} 
We can samely prove the assertion as (\hyperlink{main'2}{2}). 

\end{parts} 

Thus we have completes all the proof. 
\end{proof}

\begin{Remark} 
By Lemma \ref{ramify}, we may replace the field $K$ in Theorem \ref{explicit} with such a field $K'$ that  $I_N(K')=(1,\infty)$ and all $p_i$ and $q_i$ do not ramify in any number field included in $K'$. 
\end{Remark}

\section{Comments on non-positive weights}\label{negative} 
As we mentioned in Section \ref{Intro}, we prove the following proposition regarding the cases of non-positive weights. 

\begin{Proposition}\label{nonsense} 
Let $\g\in\R$. 
\begin{parts} 
\Part{(1)} 
If $\g<0$, then any $\g${\rm -(RB)} extension is trivial. 
\Part{(2)} 
If $\g\leq0$, then any $\g${\rm -(RN)} extension is trivial. 
\end{parts} 
\end{Proposition} 

Proposition \ref{nonsense} is an immediate consequence of the following: 

\begin{Lemma} 
Let $L/K$ be an extension of subfields of $\QB$ and $\g\in\R$. 
\begin{parts} 
\Part{(1)} 
\hypertarget{lem1}{} 
Assume that $\g<0$. 
Then $L/K$ is not $\g${\rm -(RB)} if $K$ does not have $\g${\rm -(B)}. 
\Part{(2)} 
\hypertarget{lem2}{} 
Assume that $\g\leq0$. 
Then $L/K$ is not $\g${\rm -(RN)} if $K$ does not have $\g${\rm -(N)}. 
\end{parts} 
\end{Lemma} 

\begin{proof} 
\ 

\begin{parts} 
\Part{(\hyperlink{lem1}{1})} 
Since $K$ does not have $\g$-(B), there exists a pairwise distinct sequence $(a_n)_{n\in\N}\subset K$ such that $0<h_\g(a_n)\rightarrow0$ as $n\rightarrow\infty$. 
First, we prove that $\lim_{n\rightarrow\infty}\deg(a_n)=\infty$ by contradiction. 
Suppose that there exist $M>0$ and strictly increasing sequence of positive integers $(n_i)_{i\in\N}$ such that $\deg(a_{n_i})\leq M$ for all $i\in\N$. 
Then we have 
\[ 
h(a_{n_i})\leq \l(\f{M}{\deg(a_{n_i})}\r)^{-\g}h(a_{n_i})=M^{-\g}h_\g(a_{n_i})\rightarrow0 
\] 
as $i\rightarrow\infty$. 
Thus $a_{n_i}\in\QB$ is bounded degree and bounded height for all $i\in\N$. 
This contradicts the Northcott theorem (see, \eg, \cite[p.25, Theorem 1.6.8]{BG}). 
Now fix $b\in L\setminus K$. 
Replacing $b$ with $2b$ if needed, we may assume that $ba_n$ is not a root of unity for infinitely many $n\in\N$. 
Note that we have the inequalities 
\[ 
\deg(a_n) 
\leq [\Q(b,ba_n):\Q] 
\leq \deg(b)\deg(ba_n). 
\] 
Thus we get the inequalities 
\begin{align*} 
0 < h_\g(ba_n) 
&\leq \l(\f{\deg(a_n)}{\deg(b)}\r)^\g h(ba_n) \\
&\leq \l(\f{\deg(a_n)}{\deg(b)}\r)^\g(h(b)+h(a_n)) \\
&= \f{1}{\deg(b)^\g}\l(\f{h(b)}{\deg(a_n)^{-\g}}+h_\g(a_n)\r)
\end{align*} 
for infinitely many $n\in\N$. 
Since $(ba_n)_{n\in\N}$ is a pairwise distinct sequence in $L\setminus K$, we conclude that the set $L\setminus K$ does not have $\g$-(B) by letting $n\rightarrow\infty$ in the above inequalities. 

\Part{(\hyperlink{lem2}{2})} 
We can prove the assertion in a similar way to (\hyperlink{lem1}{1}). 

\end{parts} 

These complete the proof. 
\end{proof}

\begin{acknowledgement} 
The work was supported until March 2022 by JST SPRING, Grant Number JPMJSP2136. 
\end{acknowledgement}

\end{document}